\documentclass[reqno]{amsart}

\usepackage[utf8]{inputenc}
\usepackage[english]{babel}

\usepackage{soul}
\usepackage{braket}
\usepackage{amsmath}
\usepackage{amstext}
\usepackage{amssymb}
\usepackage{amsthm}
\usepackage{mathtools}
\usepackage{stmaryrd}
\usepackage{mathrsfs}
\usepackage{extarrows}
\usepackage{afterpage}
\usepackage{array}
\usepackage{microtype}
\usepackage{graphicx}
\usepackage{amsfonts}
\usepackage{units}
\usepackage[shortlabels]{enumitem}
\usepackage{url}
\usepackage{bm}
\usepackage{tikz}
\usetikzlibrary{matrix,arrows,calc,fit,cd,positioning,intersections,arrows.meta,decorations.markings,decorations.pathmorphing,shapes,calligraphy}
\usepackage{tikz-3dplot}
\usepackage{subcaption}

\usepackage{hyperref}

\usepackage{xcolor}
\hypersetup{
    pdftitle={The K(pi, 1) conjecture for Artin groups of spherical type},
    pdfauthor={Giovanni Paolini},
    colorlinks,
    linkcolor={red!50!black},
    citecolor={blue!50!black},
    urlcolor={blue!80!black}
}

\usepackage{cleveref}
\usepackage[margin=1cm, font=small]{caption}
\usepackage{colortbl}
\usepackage{framed}

\usepackage{cleveref}

\allowdisplaybreaks[4]

\numberwithin{equation}{section}
\numberwithin{figure}{section}

\theoremstyle{definition}
\newtheorem{theorem}{Theorem}[section]

\newtheorem{lemma}[theorem]{Lemma}

\newtheorem{corollary}[theorem]{Corollary}
\newtheorem{conjecture}[theorem]{Conjecture}

\newtheorem*{mainthm*}{Main Theorem}

\newcommand{\Z}{\mathbb{Z}}
\newcommand{\R}{\mathbb{R}}
\newcommand{\C}{\mathbb{C}}

\newcommand{\M}{\mathcal{M}}
\renewcommand{\S}{\mathfrak{S}}

\newcommand{\F}{\mathcal{F}}
\newcommand{\A}{\mathcal{A}}
\newcommand{\E}{\mathcal{E}}

\newcommand{\G}{\mathcal{G}}

\DeclareMathOperator{\GL}{GL}
\DeclareMathOperator{\Br}{Br}

\newcommand{\<}{\langle}
\renewcommand{\>}{\rangle}

\newcommand{\lleq}{\leq_\text{L}}
\newcommand{\rleq}{\leq_\text{R}}

\title[The $K(\pi, 1)$ conjecture for Artin groups of spherical type]{The $K(\pi, 1)$ conjecture for Artin groups \\of spherical type}

\author{Giovanni Paolini}

\begin{document}

\begin{abstract}
    In these notes, we introduce the 50-year-old $K(\pi, 1)$ conjecture alongside Coxeter and Artin groups.
    Roughly speaking, the conjecture states that the complement in $\C^n$ of a ``symmetric'' configuration of hyperplanes is a $K(\pi, 1)$ space.
    Our end goal is to present a proof of the conjecture in the so-called spherical case, where only a finite number of hyperplanes are removed, through methods from combinatorial topology.
    This proof draws inspiration from the original proof of the spherical case, which is a special case of a celebrated 1972 theorem by Pierre Deligne.
\end{abstract}

\maketitle


\section{Introduction}
\label{sec:introduction}

The $K(\pi, 1)$ conjecture, dating back to the 1960s and 1970s and usually attributed to Brieskorn, Arnol'd, Pham, and Thom, states that the orbit configuration space of any Artin group is an Eilenberg-Maclane space (or classifying space, or $K(\pi, 1)$).
The most famous special case is the space of configurations of $n$ distinct points in $\C$, which was proved to be a classifying space for the braid group $\Br_n$ by Fox and Neuwirth \cite{fox1962braid}; see \Cref{fig:configuration-space}.

\begin{figure}[h]
    \newcommand{\scale}{2.6}
    \begin{center}
        \begin{subfigure}[t]{.3\linewidth}
            \begin{tikzpicture}[scale=1.7]
                \draw[fill=black!10] (-1, -1) rectangle (1, 1);
                \node[draw, circle, inner sep=1pt, fill=black] at (0.4, 0.15) {};
                \node[draw, circle, inner sep=1pt, fill=black] at (-0.4, 0.3) {};
                \node[draw, circle, inner sep=1pt, fill=black] at (0.1, -0.25) {};
                \node at (0.8, 0.8) {$\C$};
            \end{tikzpicture}
        \end{subfigure}\qquad
        \begin{subfigure}[t]{.3\linewidth}
            \centering
            \definecolor{darkgreen}{RGB}{0,100,0}
        	\begin{tikzpicture}[scale=1.13]
        	\draw[fill=black!10] (-1, -1.5) -- (1, -1.5) -- (1.5, -2.5) -- (-0.5, -2.5) -- cycle;
            \draw[fill=black!10, draw=white] (-1, 0.5) -- (1, 0.5) -- (1.5, -0.5) -- (-0.5, -0.5) -- cycle;

        	\coordinate (A) at (-0.3, 0.1);
        	\coordinate (B) at (0.3, -0.1);
        	\coordinate (C) at (0.8, 0.1);
        	
        	\coordinate (A1) at ($(A) + (0,-2)$);
        	\coordinate (B1) at ($(B) + (0,-2)$);
        	\coordinate (C1) at ($(C) + (0,-2)$);

        	\newcommand{\braida}{(A) ($(A) + (0.1, -0.2)$) ($(A) + (0.8, -0.7)$) ($(A) + (0.8, -1.3)$) ($(A) + (0.1, -1.7)$) (A1)}
        	
        	\begin{scope}
        	\clip (-1, -1) rectangle (1, 0.1);
        	\draw [white, line width=1mm] plot [smooth] coordinates {\braida};
        	\draw [red] plot [smooth] coordinates {\braida};
        	\end{scope}
        	
        	\newcommand{\braidb}{(B) ($(B) + (-0.2, -0.5)$) ($(B) + (0.4, -1.3)$) ($(C) + (0, -1.8)$) (C1)}
        	\draw [white, line width=1mm] plot [smooth] coordinates {\braidb};
        	\draw [blue] plot [smooth] coordinates {\braidb};
        	
        	\newcommand{\braidc}{(C) ($(C) + (0, -0.3)$) ($(C) + (-0.5, -1.3)$) ($(B) + (0, -1.7)$) (B1)}
        	\draw [white, line width=1mm] plot [smooth] coordinates {\braidc};
        	\draw [darkgreen] plot [smooth] coordinates {\braidc};
        	
        	\begin{scope}
        	\clip (-1, -2) rectangle (1, -1);
        	\draw [white, line width=1mm] plot [smooth] coordinates {\braida};
        	\draw [red] plot [smooth] coordinates {\braida};
        	\end{scope}

        	\begin{scope}[every node/.style={fill, circle, inner sep=0.9}]
        	\node at (A) {};
        	\node at (B) {};
        	\node at (C) {};
        	
        	\node at (A1) {};
        	\node at (B1) {};
        	\node at (C1) {};
        	\end{scope}

            \begin{scope}
                \draw[clip] (-1, 0.5) -- (1, 0.5) -- (1.5, -0.5) -- (-0.5, -0.5) -- cycle;
            	\draw[black!10, line width=1mm] (-1, 0.5) -- (1, 0.5) -- (1.5, -0.5) -- (-0.5, -0.5) -- cycle;
            	\draw (-1, 0.5) -- (1, 0.5) -- (1.5, -0.5) -- (-0.5, -0.5) -- cycle;
            \end{scope}
        	
        	\node at (1.35, 0.35) {$\C$};
        	\node at (1.35, -1.65) {$\C$};

        	\end{tikzpicture}
        \end{subfigure}
    \end{center}
    
    \caption{On the left, a configuration of $3$ distinct points in $\C$. The space of all such configurations has a fundamental group given by the braid group on $3$ strands: one closed loop is shown on the right, and its homotopy class forms a braid.}
    \label{fig:configuration-space}
\end{figure}
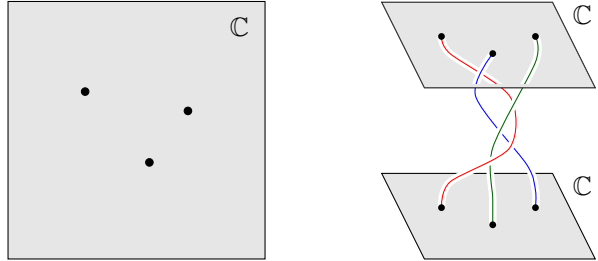

The present notes are based on a mini-course held by the author at WinterBraids 2023 in Tours.
After providing the necessary foundations on Coxeter and Artin groups, our main goal is to describe a proof of the $K(\pi, 1)$ conjecture for Artin groups of spherical type (which include the braid group as a special case), using the modern language of discrete Morse theory.
This proof is a bit different from the original proof by Deligne \cite{deligne1972immeubles} (although the core idea is the same), and draws on a survey paper by Paris \cite{paris2012k} and the author's master's thesis \cite{paolini2015thesis}.

\smallskip

\noindent\textbf{Acknowledgments.}
The author thanks the anonymous referee for their valuable suggestions.
Support from PRIN 2022A7L229 ``Algebraic and Topological Combinatorics'', as well as from INdAM's GNSAGA group, is gratefully acknowledged.

\section{Coxeter groups}

A Coxeter system is a pair $(W, S)$ where $S$ is a finite set and $W$ is a group with a presentation of the following form:
\begin{equation}
    \label{eq:coxeter-presentation}
 	W = \< S \mid s^2 = 1 \;\,\forall\, s \in S, \text{ and} \!\!\!\underbrace{stst\dotsm}_{m(s,t) \text{ terms}} \!\! = \!\! \underbrace{tsts\dotsm}_{m(s,t) \text{ terms}} \forall\, s\neq t 
    \text{ with } m(s,t) \neq \infty \;
    \>
\end{equation}
for some fixed numbers $m(s,t) = m(t, s) \in \{2, 3, 4, \dots, \infty\}$ for all $s \neq t$. The value $m(s, t) = \infty$ conventionally means no relation between $s$ and $t$.
The group $W$ is called a Coxeter group.
Whenever we talk about a Coxeter group $W$, we implicitly fix a Coxeter system $(W, S)$.

Standard textbooks about Coxeter groups are \cite{bourbaki1968elements, humphreys1992reflection, bjorner2006combinatorics}.
The most important example of a Coxeter group is the symmetric group $\S_n$ with its generating set $S = \{(1\ 2),\, (2\ 3),\, \dots, (n-1 \;\, n)\}$.
If we call $s_i$ the $i$-th generator, namely the transposition $(i \;\; i+1)$, then we can indeed observe that the following relations are satisfied:
\begin{itemize}
    \item $s_i^2 = 1$;
    \item $s_i s_{i+1}s_i = s_{i+1} s_i s_{i+1} = (i \;\, i+2)$;
    \item $s_i s_j = s_j s_i = (i \;\, i+1)(j \;\, j+1)$ whenever $|i - j| \geq 2$.
\end{itemize}
These relations turn out to present exactly the symmetric group $\S_n$, and not some larger group.
For instance, the symmetric group $\S_3$ has the following Coxeter presentation:
\[ \S_3 = \< a, b \mid a^2 = b^2 = 1, \, aba = bab \>, \]
where $a=s_1 = (1 \ 2)$ and $b = s_2 = (2 \ 3)$.

The next fundamental example of Coxeter groups is given by \textit{real reflection groups}, namely the discrete groups of Euclidean isometries of $\R^n$ generated by orthogonal reflections.
Finite reflection groups must fix at least one point of $\R^n$; they are also called \textit{spherical} because they are groups of isometries of a sphere $S^{n-1}$ centered at a fixed point.
Infinite reflection groups are also called \textit{affine}.
The collection of all fixed hyperplanes of reflections in a reflection group $W$ is called the \textit{hyperplane arrangement} associated with $W$.
See \Cref{fig:reflection-groups}.

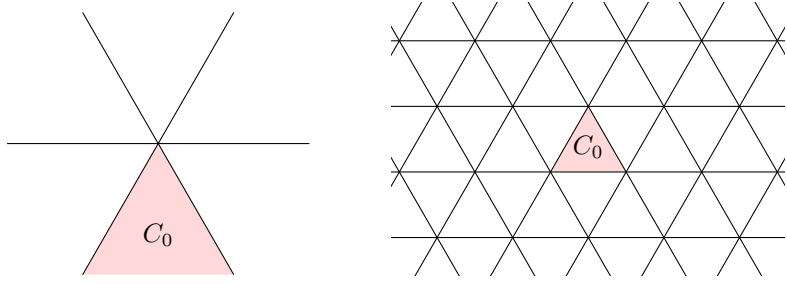
\begin{figure}
    \begin{center}
        \begin{subfigure}[t]{.3\linewidth}
            \begin{tikzpicture}[scale=0.8]
            \fill[red!15] (0, 0) -- (1.25, -2.165063509) -- (-1.25, -2.165063509) -- cycle;
        	\draw (-2.5,0) -- (2.5,0);
        	\draw[rotate=60] (-2.5,0) -- (2.5,0);
        	\draw[rotate=120] (-2.5,0) -- (2.5,0);
            \node at (0, -1.5) {$C_0$};
        	\end{tikzpicture}
        \end{subfigure}\qquad\quad
        \begin{subfigure}[t]{.45\linewidth}
            \centering
            \newcommand*\rows{10}
        	\begin{tikzpicture}[scale=1, extended line/.style={shorten >=-10cm, shorten <=-10cm}]
        	\clip (0.9,2.1) rectangle + (5.2,3.6);
            \fill[red!15] (3, 3.47) -- (4, 3.47) -- (3.5, 4.336025404) -- cycle;
        	\foreach \row in {-\rows, ...,\rows} {
        		\draw [extended line] ($\row*(0.5, {0.5*sqrt(3)})$) -- ($(\rows,0)+\row*(-0.5, {0.5*sqrt(3)})$);
        		\draw [extended line] ($\row*(1, 0)$) -- ($(\rows/2,{\rows/2*sqrt(3)})+\row*(0.5,{-0.5*sqrt(3)})$);
        		\draw [extended line] ($\row*(1, 0)$) -- ($(\rows/2,{-\rows/2*sqrt(3)})+\row*(0.5,{0.5*sqrt(3)})$);
        	}
            \node at (3.5, 3.8) {$C_0$};
        	\end{tikzpicture}
        \end{subfigure}
    \end{center}
    
    \caption{Arrangements of lines associated with two real reflection groups in $\R^2$: the symmetric group $\S_3$ on the left (also known as type $A_2$), and the affine symmetric group of type $\tilde A_2$ on the right.
    A chamber $C_0$ is highlighted. The dihedral angles between the walls of $C_0$ are all $\pi/3$, giving rise to relations such as $aba=bab$.}
    \label{fig:reflection-groups}
\end{figure}

Note that the symmetric group $\S_n$ is a reflection group: it acts on $\R^n$ by permuting the coordinates, and transpositions act as orthogonal reflections with respect to the diagonal hyperplanes $\{x_i = x_j\}$.
Since all these hyperplanes contain the line spanned by $(1, \dots, 1)$, one often restricts this representation to the codimension-one subspace orthogonal to $(1, \dots, 1)$.
For $n=3$, one obtains the arrangement in $\R^2$ depicted on the left of \Cref{fig:reflection-groups}.

\begin{theorem}[{\cite{coxeter1934discrete, witt1941spiegelungsgruppen}}]
    Every real reflection group is a Coxeter group.
\end{theorem}

\begin{proof}[Sketch of proof]
    Given a real reflection group $W$ acting on $\R^n$, one can construct a Coxeter presentation as follows. Fix any \textit{chamber} $C_0$, i.e., a connected component of the complement of the hyperplane arrangement associated with $W$ (see \Cref{fig:reflection-groups}).
    The codimension-one faces of $C_0$ span reflection hyperplanes $H_1, \dots, H_k$, called the \textit{walls} of the chamber $C_0$.
    Let $S = \{s_1, \dots, s_k\} \subseteq W$ be the set of orthogonal reflections with respect to the hyperplanes $H_1, \dots, H_k$.
    Since $W$ is discrete, any two distinct hyperplanes, $H_i$ and $H_j$, must either intersect at a dihedral angle of the form $\pi/m(s_i, s_j)$ for some integer $m(s_i, s_j) \geq 2$, or be parallel (in which case we set $m(s_i, s_j) = \infty$).
    The set $S$ with the numbers $m(s_i, s_j)$ yields a presentation of $W$ of the form \eqref{eq:coxeter-presentation}.
    See also \cite[Section 1.9]{humphreys1992reflection} for more details.
\end{proof}

It turns out that the class of spherical real reflection groups exactly coincides with the class of finite Coxeter groups.

\begin{theorem}[{\cite{coxeter1935complete,witt1941spiegelungsgruppen}}]
    Every finite Coxeter group is a (spherical) real reflection group.
\end{theorem}

From now on, we say that a Coxeter group is spherical if it is finite, or equivalently if it is a spherical real reflection group.
We say that a Coxeter group is affine if it is an affine real reflection group.

It turns out that every Coxeter group can be represented as a reflection group in a weak sense: the generators (and all their conjugates) act as linear reflections, but not necessarily as orthogonal reflections.
By linear reflection we mean a linear automorphism of a real vector space having order $2$ and fixing a hyperplane pointwise.

\begin{theorem}[\cite{tits1961groupes}]
    Any Coxeter group $W$ has a faithful representation (called the \textit{contragredient representation}) $\rho \colon W \hookrightarrow \GL(V)$ with $V = \R^{|S|}$, satisfying the following properties.
    \begin{enumerate}[(i)]
        \item For every element $r \in W$ that is conjugate to an element of $S$, we have that $\rho(r)$ is a linear reflection fixing some hyperplane $H_r \subset V$.
        \item The elements of $S$ act as reflections fixing the walls of a chamber $C_0$ of the arrangement $\A = \{ H_r \}$ ($C_0$ is often called the \textit{fundamental chamber}).
        \item Consider the union of all closed chambers in the orbit of $\overline{C_0}$:
        \[ I = \bigcup_{u \in W} \rho(u)(\overline{C_0}). \]
        Then $I \subseteq V$ is a convex cone, called the \textit{Tits cone}.
        \item $W$ acts simply transitively on the set of chambers contained in the Tits cone.
        In other words, for $u \in W$, we have that $\rho(u)(C_0) = C_0$ if and only if $u = 1$.
        \item If $W$ is irreducible, then the Tits cone $I$ coincides with the entire space $V$ if and only if $W$ is spherical; it is an open half-space, with the origin added, if and only if $W$ is affine; it does not contain any affine line if $W$ is neither spherical nor affine.
    \end{enumerate}
    \label{thm:contragredient-representation}
\end{theorem}

If $W$ is an irreducible affine Coxeter group, its contragredient representation restricts to any affine hyperplane $H$ fully contained in the Tits cone $I$. The linear reflections with respect to the hyperplanes $H_r$ restrict to affine orthogonal reflections in $H$.
This makes $W$ an affine real reflection group, as defined earlier.
An example is shown in \Cref{fig:tits-cone}.

\begin{figure}
    \begin{center}
        \begin{tikzpicture}[scale=1.5]
            \clip (-3.75, -0.75) rectangle (3.75, 1.15);
            \draw[dashed] (-5, 0) -- (5, 0);
            \node[fill=white] at (-3, 0.07) {$\vdots$};
            \node[fill=white] at (3, 0.07) {$\vdots$};
            
            \fill[red!35] (-5, 0) rectangle (5, 2);
            \fill[red!15] (0, 0) -- (1, 2) -- (-1, 2) -- cycle;

            \node at (-3, 0.07) {$\vdots$};
            \node at (3, 0.07) {$\vdots$};

            \draw[very thick, purple] (-5, 1) -- (5, 1);
            \foreach \x in {-13, -11, ..., 13} {
            	\draw (\x,2) -- (-\x,-2);
                \node[fill=white, draw=purple, circle, inner sep=1.3pt] at (\x/2, 1) {};
            }
            \node at (0, 0.6) {$C_0$};

        \end{tikzpicture}
    \end{center}
    
    \caption{The contragredient representation of the affine Coxeter group of type $\tilde A_1$, defined as $W = \< a, b \mid a^2 = b^2 = 1 \>$, also known as the infinite dihedral group.
    Here $m(a, b) = \infty$, so there is no relation between $a$ and $b$. The chamber $C_0$ is highlighted in light red, and the rest of the Tits cone $I$ is in dark red.
    The action restricts to any horizontal affine line (such as the purple one), making $W$ an affine real reflection group acting on $\R$.}
    \label{fig:tits-cone}
\end{figure}
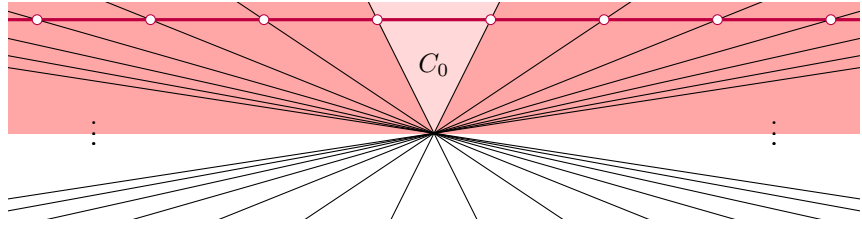

Denote by $l(u)$ the minimal length of an expression of an element $u \in W$ as a product of generators.
If $W$ is spherical, then the length function is bounded and there is a unique element of maximal length.

\begin{theorem}
    Let $W$ be a spherical Coxeter group.
    Then there is exactly one element $\delta \in W$, called the \textit{longest element}, whose length $l(\delta)$ is maximized.
    It satisfies $l(u) + l(u^{-1}\delta) = l(\delta)$ for every $u \in W$.
    \label{thm:longest-element}
\end{theorem}

\begin{proof}[Sketch of proof]
    By simple transitivity of the action of $W$ on the chambers, there exists a unique element $\delta \in W$ such that $\delta(C_0) = -C_0$, where $-C_0$ is the chamber opposite to $C_0$.
    The chamber $-C_0$ belongs to the Tits cone because $W$ is spherical, and is the unique chamber of maximal distance from $C_0$ (where the distance between two chambers is defined as the number of reflection hyperplanes separating them). This can be shown to imply that $\delta$ is the unique element of maximal length.
\end{proof}

For every $T \subseteq S$, the elements of $T$ generate a subgroup $W_T \subseteq W$ called a \textit{standard parabolic subgroup}.
We end this section with two important properties of parabolic subgroups and their cosets.

\begin{theorem}
    Let $T \subseteq S$.
    \begin{enumerate}[(i)]
        \item The pair $(W_T, T)$ is a Coxeter system.
        More precisely, a Coxeter presentation for the subgroup $W_T$ can be naturally obtained from \eqref{eq:coxeter-presentation} by removing all generators except those in $T$ and all relations except those involving elements of $T$.

        \item Every (left) coset of $W_T$ in $W$ has a unique element of smallest length $u_0$, called its \textit{minimal coset representative}.
        It satisfies $l(u_0u) = l(u_0) + l(u)$ for every $u \in W_T$.
    \end{enumerate}
    \label{thm:parabolic-subgroups}
\end{theorem}

\section{Artin groups}

To every Coxeter group $W$ presented as in \eqref{eq:coxeter-presentation}, there is an associated Artin group defined as follows:
\begin{equation}
    \label{eq:artin-presentation}
 	G_W = \< S \mid \underbrace{stst\dotsm}_{m(s,t) \text{ terms}} \! = \! \underbrace{tsts\dotsm}_{m(s,t) \text{ terms}} \forall\, s\neq t
    \text{ with } m(s,t) \neq \infty
    \>.
\end{equation}
While there is no standard reference about Artin groups, a recommended read is a survey by Paris \cite{paris2012k}, which develops many of the topics covered in these notes in greater detail.

Two basic yet important examples of Artin groups are the free group $F_n$ (obtained by setting all $m(s, t) = \infty$) and the free abelian group $\Z^n$ (obtained by setting all $m(s, t) = 2$).
However, the example that started the whole theory of Artin groups is the \textit{braid group}, introduced by Artin in 1925 \cite{artin1925theorie}.
The braid group $\Br_n$ is the Artin group associated with the symmetric group $\S_n$.
For instance, the braid group $\Br_3$ has the following presentation:
\[ \Br_3 = \< a, b \mid aba = bab \>. \]
Elements of $\Br_n$ can be interpreted as isotopy classes of braids on $n$ strands, as shown in \Cref{fig:configuration-space}.
Artin groups associated with spherical (resp., affine) Coxeter groups are called of spherical type (resp., affine type).
For instance, the braid group $\Br_n$ and the free abelian group $\Z^n$ are Artin groups of spherical type; the free group $F_2$ is an Artin group of affine type, as it is associated with the affine Coxeter group of type $\tilde A_1$ (introduced in \Cref{fig:tits-cone}).

The contragredient representation of a Coxeter group $W$ (\Cref{thm:contragredient-representation}) gives rise to the \textit{orbit configuration space} of $W$:
\[ Y_W = \left( (I + iV) \setminus \bigcup_r \, H_r \otimes_\R \C \right) / \, W. \]
Here, $I + iV$ is the subset of the complex vector space $V \otimes_\R \C$ consisting of all vectors with real part belonging to the Tits cone $I$; for each reflection $r \in W$, the set $H_r \otimes \C \subset V \otimes_\R \C$ is the complex hyperplane with the same (real) equation as $H_r$; the Coxeter group $W$ acts diagonally on the real and the imaginary parts of vectors in $V \otimes_\R \C$, and we are taking the quotient by this action.
Since the (complexified) reflection hyperplanes are removed, the action of $W$ is a covering space action.

For $W = \S_n$, the orbit configuration space is
\[ Y_{\S_n} = \{ (z_1, \dots, z_n) \in \C^n \mid z_i \neq z_j \; \forall \, i \neq j \} / \, \S_n.   \]
This is exactly the space of configurations of $n$ distinct points in $\C$; note that the quotient by the $\S_n$ action has the effect of forgetting the order of the $n$ points.
As illustrated in \Cref{fig:configuration-space}, the fundamental group of $Y_{\S_n}$ is the braid group $\Br_n$, and this is, in fact, true more generally if we replace $\S_n$ with an arbitrary Coxeter group.

\begin{theorem}[\cite{brieskorn1971fundamentalgruppe, van1983homotopy}]
    For every Coxeter group $W$, we have an isomorphism $\pi_1(Y_W) \cong G_W$.
\end{theorem}

Fox and Neuwirth proved that $Y_{\S_n}$ is a classifying space for the braid group $\Br_n$ \cite{fox1962braid}.
This led to a formulation of the following ``$K(\pi, 1)$ conjecture,'' which is one of the most important open problems about Artin groups.

\begin{conjecture}
    For every Coxeter group $W$, we have $Y_W \simeq K(G_W, 1)$.
\end{conjecture}

The $K(\pi, 1)$ conjecture is known for instance in the following cases: spherical \cite{deligne1972immeubles}, affine \cite{paolini2021proof}, two-dimensional and FC-type \cite{charney1995k}. See also \cite{brieskorn1973groupes, okonek1979dask, callegaro2010k, hendriks1985hyperplane, delucchi2022dual, haettel2023new, huang2023labeled, huang2024cycles}.

The orbit configuration space $Y_W$ has the homotopy type of a CW complex with a finite number of cells.
This was proved by Salvetti \cite{salvetti1987topology, salvetti1994homotopy}, and the complex he constructed is known as the Salvetti complex.

\begin{theorem}[\cite{salvetti1994homotopy}]
    The orbit configuration space $Y_W$ is homotopy equivalent to a CW complex $X_W$, called the Salvetti complex, having one cell $\sigma_T$ of dimension $|T|$ for every subset $T \subseteq S$ such that the standard parabolic subgroup $W_T \subseteq W$ is spherical (i.e., finite).
\end{theorem}

The cells $\sigma_T$ of the Salvetti complex can be described geometrically as follows.
Given a spherical standard parabolic subgroup $W_T$, consider its representation as a spherical real reflection group acting on $\R^n$. Fix any point $x$ in the fundamental chamber $C_0$, and denote by $Q \subseteq \R^n$ the polytope obtained as the convex hull of the points $u(x)$ for $u \in W_T$. The polytope $Q$ is dual to the hyperplane arrangement associated with $W_T$, as shown in \Cref{fig:salvetti-cell} on the left.
Identify the cell $\sigma_T$ with $Q$.

For every face $F \subset Q$, there exist a subset $R \subset T$ and a minimal coset representative $u_0\in W_T$ for the coset $u_0 W_R \subset W_T$ such that $F$ is the convex hull of the vertices $u_0 u (x)$ for $u \in W_R$.
Then $F$ can be canonically identified with the cell $\sigma_R$, as illustrated in \Cref{fig:salvetti-cell} on the right.
This construction describes how $\sigma_T$ is attached to the lower-dimensional cells $\sigma_R$ for $R \subset T$ in the Salvetti complex $X_W$.

Note that, for every standard parabolic subgroup $W_T$, the Salvetti complex $X_{W_T}$ is naturally a subcomplex of the Salvetti complex $X_W$.
Geometrically, this corresponds to an inclusion of $Y_{W_T}$ as a subspace of $Y_W$ in a neighborhood of a point fixed by $W_T$ (see e.g.\ \cite[Section 6]{lofano2021euclidean}).

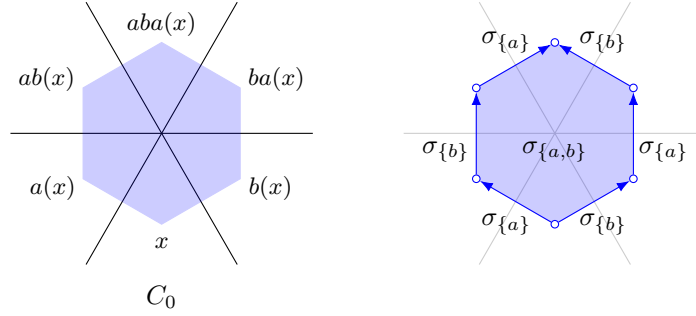
\begin{figure}
    \begin{center}
        \begin{subfigure}[t]{.3\linewidth}
            \begin{tikzpicture}[scale=0.8]
        	\draw (-2.5,0) -- (2.5,0);
        	\draw[rotate=60] (-2.5,0) -- (2.5,0);
        	\draw[rotate=120] (-2.5,0) -- (2.5,0);
            \node at (0, -2.7) {$C_0$};
            \draw[fill=blue, blue, opacity=0.2] (0, -1.5) -- (-1.299038106, -0.75) -- (-1.299038106, 0.75) -- (0, 1.5) -- (1.299038106, 0.75) -- (1.299038106, -0.75) -- cycle;
            {\small
            \node at (0, -1.8) {$x$};
            \node at (-1.8, -0.9) {$a(x)$};
            \node at (1.8, -0.9) {$b(x)$};
            \node at (-1.9, 0.9) {$ab(x)$};
            \node at (1.9, 0.9) {$ba(x)$};
            \node at (0, 1.8) {$aba(x)$};
            }
        	\end{tikzpicture}
        \end{subfigure}\qquad\qquad
        \begin{subfigure}[t]{.3\linewidth}
            \centering
            \begin{tikzpicture}[scale=0.8]
        	\draw[black!20] (-2.5,0) -- (2.5,0);
        	\draw[rotate=60, black!20] (-2.5,0) -- (2.5,0);
        	\draw[rotate=120, black!20] (-2.5,0) -- (2.5,0);
            \node at (0, -2.7) {\phantom{$C_0$}};
            \draw[fill=blue, blue, opacity=0.2] (0, -1.5) -- (-1.299038106, -0.75) -- (-1.299038106, 0.75) -- (0, 1.5) -- (1.299038106, 0.75) -- (1.299038106, -0.75) -- cycle;

            \begin{scope}[every node/.style={fill=white, circle, draw=blue, inner sep=1}]
                \node (S) at (0, -1.5) {};
                \node (N) at (0, 1.5) {};
                \node (SW) at (-1.299038106, -0.75) {};
                \node (SE) at (1.299038106, -0.75) {};
                \node (NW) at (-1.299038106, 0.75) {};
                \node (NE) at (1.299038106, 0.75) {};
            \end{scope}

            \draw[blue, -Latex] (S) -- (SW);
            \draw[blue, -Latex] (S) -- (SE);

            \draw[blue, -Latex] (SW) -- (NW);
            \draw[blue, -Latex] (SE) -- (NE);
            
            \draw[blue, -Latex] (NW) -- (N);
            \draw[blue, -Latex] (NE) -- (N);

            \node at (0, -0.3) {$\sigma_{\{a, b \}}$};
            \node at (-0.8, -1.5) {$\sigma_{\{a\}}$};
            \node at (0.8, -1.5) {$\sigma_{\{b\}}$};
            \node at (-1.8, -0.3) {$\sigma_{\{b\}}$};
            \node at (1.8, -0.3) {$\sigma_{\{a\}}$};
            \node at (-0.8, 1.5) {$\sigma_{\{a\}}$};
            \node at (0.8, 1.5) {$\sigma_{\{b\}}$};

         \end{tikzpicture}
        \end{subfigure}
    \end{center}
    
    \caption{On the left, the polytope dual to the hyperplane arrangement associated with the symmetric group $\S_3 = \< a, b \mid a^2 = b^2 = 1, \; aba = bab \>$.
    On the right, the cell $\sigma_{T}$ for $T = \{a, b\}$ when $m(a, b) = 3$. The arrows point away from the minimal coset representative, and describe how the $1$-dimensional faces with the same label ($\sigma_{\{a\}}$ or $\sigma_{\{b\}}$) are identified. All six $0$-dimensional faces are identified with $\sigma_\varnothing$.}
    \label{fig:salvetti-cell}
\end{figure}

\section{Artin monoids}

While Artin groups are fairly complicated to understand, the closely related \emph{Artin monoids} are more approachable.
For every Coxeter system $(W, S)$, the associated Artin monoid $G_W^+$ is defined through a monoid presentation identical to the Artin group presentation \eqref{eq:artin-presentation}:
\begin{equation}
 	G_W^+ = \< S \mid \underbrace{stst\dotsm}_{m(s,t) \text{ terms}} \! = \! \underbrace{tsts\dotsm}_{m(s,t) \text{ terms}} \forall\, s\neq t 
    \text{ with } m(s,t) \neq \infty
    \>_+
    \nonumber
\end{equation}
(the ``$+$'' on the right-hand side denotes a monoid presentation).
There is a natural monoid map $G_W^+ \to G_W$ sending each generator of the Artin monoid $G_W^+$ to its copy inside the Artin group $G_W$.
It is not obvious, but true, that an Artin monoid injects into its Artin group.

\begin{theorem}[\cite{brieskorn1972artin, deligne1972immeubles, paris2002artin}]
    For every Coxeter group $W$, the natural map $G_W^+ \to G_W$ is injective.
\end{theorem}

For example, the element $aba = bab \in \Br_3$ also belongs to the Artin monoid $\Br_3^+$, whereas the element $a b^{-1}$ does not, as it cannot be written as a product of the generators (without using their inverses).

Since the defining relations are homogeneous, every word (over the alphabet $S$) representing a given monoid element $\alpha \in G_W^+$ has the same length, and we denote it by $l(\alpha)$.
In addition, every element $u$ of the Coxeter group $W$ can be lifted to an Artin monoid element by taking any expression of minimal length $u = s_{i_1} s_{i_2} \cdots s_{i_k}$ and reading it as a word in the Artin monoid $G_W^+$.
Call $\iota \colon W \to G_W^+$ the resulting function, which is a section of the natural projection $\pi \colon G_W^+ \to W$.

\begin{theorem}[\cite{tits1969probleme}]
    The section $\iota \colon W \to G_W^+$ is well-defined, i.e., it does not depend on the chosen minimal-length expression.
\end{theorem}

The section $\iota$ preserves the length. However, it is not a monoid homomorphism.
The previous theorem essentially states that any two minimal-length expressions of an element $u \in W$ can be transformed into each other through a sequence of substitutions $stst\cdots \leftrightarrow tsts \cdots$ (where both subwords have $m(s,t)$ factors).
This is known as the \textit{Matsumoto property} and it is closely related to the \textit{exchange property}.

In the Artin monoid, we have divisibility relations on the left and the right: we say that $\alpha$ left-divides $\beta$ (resp.\ $\alpha$ right-divides $\beta$) if $\beta = \alpha\gamma$ (resp.\ $\beta = \gamma\alpha$) for some $\gamma \in G_W^+$; when this happens, we also write $\alpha \lleq \beta$ (resp.\ $\alpha \rleq \beta$).

\begin{theorem}[\cite{brieskorn1972artin}]
    Let $W$ be any Coxeter group.
    \begin{enumerate}[(i)]
        \item Any nonempty subset $E \subseteq G_W^+$ has a (left or right) greatest common divisor.
        \item A subset $T\subseteq S$ admits a (left or right) least common multiple in $G_W^+$ if and only if the standard parabolic subgroup $W_T \subseteq W$ is spherical.
        When this happens, the left and right least common multiples coincide and are denoted by $\Delta_T$ (the \textit{Garside element} of the Artin group $G_{W_T}$).
    \end{enumerate}
    \label{thm:gcd-lcm}
\end{theorem}

The Garside element of an Artin group of spherical type enjoys many useful properties, some of which are presented in the following theorem.

\begin{theorem}[\cite{brieskorn1972artin}]
    Let $W$ be a spherical Coxeter group and $\Delta \in G_W^+$ the Garside element.
    \begin{enumerate}[(i)]
        \item $\Delta = \iota(\delta)$, where $\delta$ is the longest element of the Coxeter group $W$.
        \item Every element $\alpha \in G_W$ can be written in the form $\alpha = \Delta^k \beta$ for some $\beta \in G_W^+$ and $k \in \Z$.
        This expression is unique if we require $k$ to be as large as possible.
    \end{enumerate}
\end{theorem}

\section{The $K(\pi, 1)$ conjecture for Artin groups of spherical type}

In this section, we outline a proof of the $K(\pi, 1)$ conjecture for Artin groups of spherical type.
To do this, we need to show that the universal cover $\tilde X_W$ of the Salvetti complex $X_W$ is contractible when $W$ is a spherical Coxeter group.

The universal cover $\tilde X_W$ inherits a cellular structure from $X_W$; the cells of $\tilde X_W$ are indexed by pairs $(\alpha, T)$ for $\alpha \in G_W = \pi_1(X_W)$ and $T \subseteq S$ with $W_T$ finite.
The boundary of a cell $\sigma_{(\alpha, T)}$ of $\tilde X_W$ consists of all the cells $\sigma_{(\alpha \, \iota(u_0), R)}$ for all cosets $u_0W_R \subseteq W_T$ (here $u_0$ is the minimal representative of the coset $u_0 W_R$).
These cells are all distinct, and $\tilde X_W$ is a regular CW complex.
See \Cref{fig:salvetti-cover}.

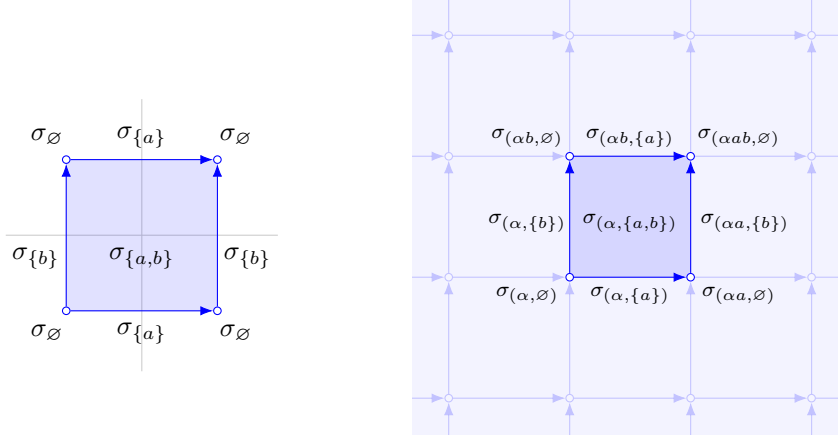
\begin{figure}
    \begin{center}
        \begin{subfigure}[t]{.4\linewidth}
            \begin{tikzpicture}[scale=1]
            \draw[black!20] (-1.8,0) -- (1.8,0);
            \draw[black!20] (0, -1.8) -- (0, 1.8);

            \draw[fill=blue, blue, opacity=0.12] (-1, -1) rectangle (1, 1);

            \begin{scope}[every node/.style={fill=white, circle, draw=blue, inner sep=1}]
                \node (SW) at (-1, -1) {};                
                \node (SE) at (1, -1) {};                
                \node (NW) at (-1, 1) {};                
                \node (NE) at (1, 1) {};                
            \end{scope}

            \draw[blue, -Latex] (SW) -- (SE);
            \draw[blue, -Latex] (SW) -- (NW);
            \draw[blue, -Latex] (NW) -- (NE);
            \draw[blue, -Latex] (SE) -- (NE);

            \node at (0, -0.3) {$\sigma_{\{a, b \}}$};
            \node at (0, -1.3) {$\sigma_{\{a\}}$};
            \node at (0, 1.3) {$\sigma_{\{a\}}$};
            \node at (-1.4, -0.3) {$\sigma_{\{b\}}$};
            \node at (1.4, -0.3) {$\sigma_{\{b\}}$};

            \node at (-1.25, -1.3) {$\sigma_{\varnothing}$};
            \node at (1.25, -1.3) {$\sigma_{\varnothing}$};
            \node at (-1.25, 1.3) {$\sigma_{\varnothing}$};
            \node at (1.25, 1.3) {$\sigma_{\varnothing}$};

            \node at (0, -2.5) {\phantom{.}};
            \end{tikzpicture}
        \end{subfigure}\quad
        \begin{subfigure}[t]{.5\linewidth}
            \begin{tikzpicture}[scale=1.6]
                \clip (-1.3, -1.3) rectangle (2.3, 2.3);
                \draw[fill=blue, blue, opacity=0.05] (-2, -2) rectangle (3, 3);
                \draw[fill=blue, blue, opacity=0.12] (0, -0) rectangle (1, 1);

                \foreach \x in {-2, ..., 3} {
                    \foreach \y in {-2, ..., 3} {
                        \pgfmathparse{
                            (\x==0 && \y==0) || 
                            (\x==1 && \y==0) || 
                            (\x==0 && \y==1) || 
                            (\x==1 && \y==1)
                        }
                        \ifnum\pgfmathresult=1
                            \node[fill=white, circle, draw=blue, inner sep=1] (X-\x-\y) at (\x, \y) {};
                        \else
                            \node[fill=white, circle, draw=blue, inner sep=1, opacity=0.2] (X-\x-\y) at (\x, \y) {};
                        \fi
                    }
                }

                \foreach \x in {-2,...,2} {
                    \foreach \y in {-2,...,3} {
                        \pgfmathparse{
                            !((\x==0 && \y==0 && \x+1==1 && \y==0) || 
                            (\x==1 && \y==0 && \x+1==0 && \y==0) ||
                            (\x==0 && \y==1 && \x+1==1 && \y==1) ||
                            (\x==1 && \y==1 && \x+1==0 && \y==1))
                        }
                        \ifnum\pgfmathresult=1
                            \draw[blue, -Latex, opacity=0.2] (X-\x-\y) -- (X-\the\numexpr\x+1\relax-\y);
                        \else
                            \draw[blue, -Latex] (X-\x-\y) -- (X-\the\numexpr\x+1\relax-\y);
                        \fi
                    }
                }

                \foreach \x in {-2,...,3} {
                    \foreach \y in {-2,...,2} {
                        \pgfmathparse{
                            !((\x==0 && \y==0 && \x==0 && \y+1==1) || 
                            (\x==1 && \y==0 && \x==1 && \y+1==1) ||
                            (\x==0 && \y==1 && \x==0 && \y+1==0) ||
                            (\x==1 && \y==1 && \x==1 && \y+1==0))
                        }
                        \ifnum\pgfmathresult=1
                            \draw[blue, -Latex, opacity=0.2] (X-\x-\y) -- (X-\x-\the\numexpr\y+1\relax);
                        \else
                            \draw[blue, -Latex] (X-\x-\y) -- (X-\x-\the\numexpr\y+1\relax);
                        \fi
                    }
                }

                {
                \footnotesize
                    \node at (-0.35, -0.15) {$\sigma_{(\alpha, \varnothing)}$};
                    \node at (-0.35, 1.15) {$\sigma_{(\alpha b, \varnothing)}$};
                    \node at (1.4, -0.15) {$\sigma_{(\alpha a, \varnothing)}$};
                    \node at (1.4, 1.15) {$\sigma_{(\alpha ab, \varnothing)}$};

                    \node at (0.5, -0.15) {$\sigma_{(\alpha, \{a\})}$};
                    \node at (-0.35, 0.45) {$\sigma_{(\alpha, \{b\})}$};
                    \node at (0.5, 1.15) {$\sigma_{(\alpha b, \{a\})}$};
                    \node at (1.45, 0.45) {$\sigma_{(\alpha a, \{b\})}$};

                    \node at (0.5, 0.45) {$\sigma_{(\alpha, \{a, b\})}$};
                }
            \end{tikzpicture}
        \end{subfigure}
    \end{center}
    
    \caption{On the left, the Salvetti complex $X_W$ for the Coxeter group $W = \Z_2 \times \Z_2 = \< a, b \mid a^2 = b^2 = 1, \, ab = ba \>$. In this case, $X_W$ is a torus $S^1 \times S^1$ with its minimal cell structure.
    On the right, the universal cover $\tilde X_W \cong \R^2$. For every $\alpha \in G_W = \Z \times \Z$, the universal cover has: one $2$-cell $\sigma_{(\alpha, \{a, b\})}$; two $1$-cells $\sigma_{(\alpha, \{a\})}$ and $\sigma_{(\alpha, \{b\})}$; and one $0$-cell $\sigma_{(\alpha, \varnothing)}$. A generic $2$-cell $\sigma_{(\alpha, \{a, b\})}$ is highlighted, together with all the cells in its boundary.}
    \label{fig:salvetti-cover}
\end{figure}

The universal cover $\tilde X_W$ has a subcomplex $\tilde X_W^+$ consisting of all cells $\sigma_{(\alpha, T)}$ such that $\alpha \in G_W^+$.
We call $\tilde X_W^+$ the positive subcomplex of $\tilde X_W$.

\begin{lemma}
    If $W$ is spherical and $\tilde X_W^+$ is contractible, then $\tilde X_W$ is also contractible.
    \label{lemma0}
\end{lemma}

\begin{proof}
    Let $\Delta \in G_W^+$ be the Garside element of $G_W$.
    If $\sigma_{(\alpha, T)}$ is any cell of $\tilde X_W$, we can write $\alpha = \Delta^{k}\beta$ for some $k \in \Z$ and $\beta \in G_W^+$.
    Then $\sigma_{(\alpha, T)}$ belongs to the subcomplex $Z_k = \Delta^{k} \tilde X_W^+ \subseteq \tilde X_W$; here we are letting the elements $\Delta^k \in G_W$ act on $\tilde X_W$ as deck transformations.

    We have therefore written $\tilde X_W$ as the union of a chain of subcomplexes $\ldots \supset Z_{-1} \supset Z_0 \supset Z_1 \supset Z_2 \supset \ldots$, all homeomorphic to $\tilde X_W^+$ and thus contractible.
    The homotopy groups of $\tilde X_W$ must vanish because any homotopy class can be represented by a map $S^n \to Z_k$ for some $k$.
    By Whitehead's theorem, the entire universal cover $\tilde X_W$ is also contractible.
\end{proof}

So it is enough to prove that the positive subcomplex $\tilde X_W^+$ is contractible, and this can be done for all Coxeter groups $W$ (not necessarily finite).
The idea is to construct a deformation retraction of $\tilde X_W^+$ onto the $0$-cell $\sigma_{(1, \varnothing)}$ through a discrete Morse vector field \cite{forman1998morse}, as illustrated in \Cref{fig:collapse}.
We outline the main results we are going to use from Forman's \textit{discrete Morse theory}, developed in \cite{forman1998morse, forman2002user, chari2000discrete, batzies2002discrete}; we refer the reader to \cite[Chapter 11]{kozlov2007combinatorial} for a textbook introduction to the topic.

\begin{figure}
    \begin{center}
        \begin{tikzpicture}[scale=1.6]
            \begin{scope}
                \clip (-0.5, -0.5) rectangle (3.3, 3.3);
                \draw[fill=blue, blue, opacity=0.05] (0, 0) rectangle (4, 4);
                \draw[fill=blue, blue, opacity=0.12] (2, 1) rectangle (3, 2);
                
                \foreach \x in {0, ..., 4} {
                    \foreach \y in {0, ..., 4} {
                        \pgfmathparse{
                            (\x==3 && \y==2)
                        }
                        \ifnum\pgfmathresult=1
                            \node[fill=white, circle, draw=blue, inner sep=1] (X-\x-\y) at (\x, \y) {};
                        \else
                            \node[fill=white, circle, draw=blue, inner sep=1, opacity=0.2] (X-\x-\y) at (\x, \y) {};
                        \fi
                    }
                }
    
                \foreach \x in {0,...,3} {
                    \foreach \y in {0,...,4} {
                        \pgfmathparse{
                            !(\x+1==3 && \y==2)
                        }
                        \ifnum\pgfmathresult=1
                            \draw[blue, opacity=0.2] (X-\x-\y) -- (X-\the\numexpr\x+1\relax-\y);
                        \else
                            \draw[blue] (X-\x-\y) -- (X-\the\numexpr\x+1\relax-\y);
                        \fi
                    }
                }
    
                \foreach \x in {0,...,4} {
                    \foreach \y in {0,...,3} {
                        \pgfmathparse{
                            !(\x==3 && \y+1==2)
                        }
                        \ifnum\pgfmathresult=1
                            \draw[blue, opacity=0.2] (X-\x-\y) -- (X-\x-\the\numexpr\y+1\relax);
                        \else
                            \draw[blue] (X-\x-\y) -- (X-\x-\the\numexpr\y+1\relax);
                        \fi
                    }
                }
    
                \foreach \x in {0, ..., 3} {
                    \foreach \y in {0, ..., 4} {
                        \pgfmathparse{
                            !(\x+1==3 && \y==2)
                        }
                    
                        \pgfmathsetmacro{\midX}{\x + 0.5}
                        \pgfmathsetmacro{\midY}{\y}
    
                        \ifnum\pgfmathresult=1
                            \draw[magenta, ->, thick, opacity=0.45] (X-\the\numexpr\x+1\relax-\y) -- (\midX, \midY);
                        \else
                            \draw[magenta, ->, thick] (X-\the\numexpr\x+1\relax-\y) -- (\midX, \midY);
                        \fi

                        \pgfmathparse{
                            !(\x+1==3 && \y+1==2)
                        }
    
                        \pgfmathsetmacro{\higherX}{\x + 1}
                        \pgfmathsetmacro{\higherY}{\y + 0.5}
    
                        \ifnum\pgfmathresult=1
                            \draw[magenta, ->, thick, opacity=0.45] (\higherX, \higherY) -- (\midX, \higherY);
                        \else
                            \draw[magenta, ->, thick] (\higherX, \higherY) -- (\midX, \higherY);
                        \fi

                    }
                }
    
                \foreach \y in {0, ..., 3} {
                    \pgfmathsetmacro{\midX}{0}
                    \pgfmathsetmacro{\midY}{\y + 0.5}
                    \draw[magenta, ->, thick, opacity=0.45] (X-0-\the\numexpr\y+1\relax) -- (\midX, \midY);
                }
            \end{scope}

            {
            \footnotesize
                \node at (-0.25, -0.12) {$\sigma_{(1, \varnothing)}$};
                \node at (3.45, 2.18) {$\sigma_{(a^3b^2, \varnothing)}$};

                \node at (2.5, 2.18) {$\sigma_{(a^2b^2, \{a\})}$};
                \node at (3.45, 1.32) {$\sigma_{(a^3b, \{b\})}$};

                \node at (2.5, 1.32) {$\sigma_{(a^2b, \{a, b\})}$};
            }

        \end{tikzpicture}
    \end{center}
    
    \caption{The positive subcomplex $\tilde X_W^+$ for the Coxeter group $\Z_2 \times \Z_2$, with arrows indicating a Morse matching (or \textit{discrete Morse vector field}) showing that $\tilde X_W^+$ deformation retracts onto the $0$-cell $\sigma_{(1, \varnothing)}$.
    The four cells in $\eta^{-1}(a^3b^2)$ are highlighted: the $0$-cell $\sigma_{(a^3b^2, \varnothing)}$ is matched with the $1$-cell $\sigma_{(a^2b^2, \{a\})}$, and the $1$-cell $\sigma_{(a^3b, \{b\})}$ is matched with the $2$-cell $\sigma_{(a^2b, \{a, b\})}$.
    }
    \label{fig:collapse}
\end{figure}
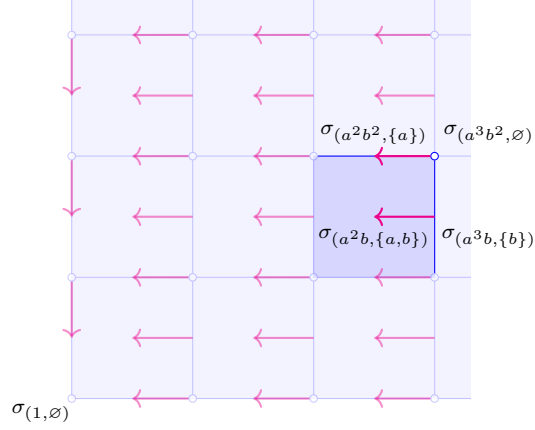

A discrete Morse vector field is encoded through a matching on the cells of a CW complex $Z$, which for simplicity we assume to be regular.
More precisely, let $\G(Z) = (\F(Z), \E(Z))$ be the directed graph where the set of vertices $\F(Z)$ consists of all cells of $Z$, and the set of edges $\E(Z)$ consists of all pairs $\sigma \to \tau$ where $\tau$ is a codimension-one face of $\sigma$.
A \emph{Morse matching} is a subset $\M \subseteq \E(Z)$ satisfying the following properties:
\begin{enumerate}
    \item Every cell appears at most once in $\M$ (so $\M$ is a partial matching on the graph $\G(Z)$).
    \item The graph $\G_\M(Z)$, obtained from $\G(Z)$ by reversing all the edges in $\M$, is acyclic.
    \item For every $\sigma \in \F(Z)$, the set of cells reachable from $\sigma$ in the graph $\G_\M(Z)$ is finite.
\end{enumerate}
The edges that are reversed when passing from $\G(Z)$ to $\G_\M(Z)$ are precisely the arrows in \Cref{fig:collapse}.
Given a Morse matching $\M$ on $Z$, the cells of $Z$ that do not appear in $\M$ are called $\M$-\emph{critical}.
The following is a special case of the main theorem of discrete Morse theory.

\begin{theorem}[\cite{forman1998morse, batzies2002discrete}]
    Let $\M$ be a Morse matching on a regular CW complex $Z$.
    If the $\M$-critical cells form a subcomplex $Z_\M$ of $Z$, then $Z$ deformation retracts onto $Z_\M$.
    \label{thm:dmt}
\end{theorem}

The assumptions that $Z$ is regular and that the $\M$-critical cells form a subcomplex can both be relaxed, but we do not need such generality in the present notes.

We shall make use of a standard tool to construct Morse matchings, often referred to as the Patchwork Theorem.
For this, let $(P, \leq)$ be any poset, and suppose we have a map $\eta \colon \F(Z) \to P$ that respects the partial ordering, i.e., such that $\eta(\tau) \leq \eta(\sigma)$ for every directed edge $\sigma \to \tau$ in the graph $\G(Z)$.
Denote by $Z_{\leq p}$ the subcomplex of $Z$ consisting of all cells $\sigma$ with $\eta(\sigma) \leq p$.

\begin{theorem}[\cite{hersh2005optimizing, jonsson2008simplicial}]
    For every $p \in P$, assume that the following properties hold.
    \begin{enumerate}
        \item There is a partial matching $\M_p \subseteq \E(Z)$ only involving the cells in the fiber $\eta^{-1}(p)$.
        \item The induced subgraph of $\G_{\M_p}(Z)$ on the vertex set $\eta^{-1}(p)$ is acyclic.
        \item The subcomplex $Z_{\leq p}$ has a finite number of cells (i.e., it is compact).
    \end{enumerate}
    Then the union $\M = \bigcup_{p \in P} \M_p$ is a Morse matching on $Z$.
    \label{thm:patchwork}
\end{theorem}

We are now ready to construct a Morse matching on the positive subcomplex $\tilde X_W^+$ for any Coxeter group $W$.
Let $\eta \colon \F(\tilde X_W^+) \to G_W^+$ be the map that sends a cell $\sigma_{(\alpha, T)}$ to its ``longest vertex'' $\alpha \Delta_T$.
To apply \Cref{thm:patchwork}, we now show that $\eta$ respects the partial ordering, where the Artin monoid $G_W^+$ is equipped with the partial order of left divisibility $\lleq$.

\begin{lemma}
    If $\sigma \to \tau$ is an edge in $\G(\tilde X_W^+)$, then $\eta(\tau) \lleq \eta(\sigma)$.
    \label{lemma1}
\end{lemma}

\begin{proof}
    Let $\sigma = \sigma_{(\alpha, T)}$ and $\tau = \sigma_{(\alpha \, \iota(u_0), R)}$, where $u_0$ is the minimal representative of the coset $u_0 W_R \subset W_T$.
    We need to prove that $\alpha\, \iota(u_0) \Delta_R \lleq \alpha\Delta_T$, which is equivalent to $\iota(u_0) \Delta_R \lleq \Delta_T$.
    In the Coxeter group $W_T$, we have $l(u_0) + l(\delta_R) = l(u_0 \delta_R)$ by \Cref{thm:parabolic-subgroups}.
    Since $\delta_T$ is the longest element of $W_T$, we can write $\delta_T = u_0\delta_R v$ with $l(u_0 \delta_R) + l(v) = l(\delta_T)$ by \Cref{thm:longest-element}.
    Applying $\iota$, we obtain the factorization $\Delta_T = \iota(u_0) \Delta_R \, \iota(v)$, so in particular $\iota(u_0) \Delta_R \lleq \Delta_T$.
\end{proof}

Next, we construct perfect acyclic matchings $\M_\beta$ on all fibers $\eta^{-1}(\beta)$ for $\beta \neq 1$ by using the results on greatest common divisors and least common multiples in Artin monoids.
The matching we construct is the one exemplified in \Cref{fig:collapse}.

\begin{lemma}
    Let $\beta \in G_W^+ \setminus \{1\}$.
    Then there exists a matching $\M_\beta \subseteq \E(\tilde X_W^+)$, involving all the vertices in $\eta^{-1}(\beta)$, such that the induced subgraph of $\G_{\M_\beta}(\tilde X_W^+)$ on the vertex set $\eta^{-1}(\beta)$ is acyclic.
    \label{lemma2}
\end{lemma}

\begin{proof}
    Let $T_\beta = \{s \in S \mid s \rleq \beta\}$.
    By \Cref{thm:gcd-lcm}, we have that $\Delta_R \rleq \beta$ for every $R \subseteq T_\beta$.
    Therefore, the fiber $\eta^{-1}(\beta)$ is naturally in bijection with the powerset of $T_\beta$, and this bijection is order-preserving: $\sigma_{(\alpha_1, R_1)}$ is a face of $\sigma_{(\alpha_2, R_2)}$ if and only if $R_1 \subseteq R_2$.
    The induced subgraph of $\G(\tilde X_W^+)$ on the vertex set $\eta^{-1}(\beta)$ is then the $1$-skeleton of a $|T_\beta|$-dimensional hypercube.
    To construct a matching $\M_\beta$, we can fix any $\bar s \in T_\beta$ (the set $T_\beta$ is non-empty because $\beta \neq 1$) and pair $R \cup \{\bar s\}$ with $R$, for every $R \subseteq T_\beta \setminus \{ \bar s \}$.
    In the induced subgraph of $\G_{\M_\beta}(\tilde X_W^+)$, edges correspond to either removing an element of $T_\beta \setminus \{\bar s\}$ or adding $\bar s$; this subgraph has no cycles.
\end{proof}

We can finally put all the ingredients together.

\begin{theorem}
    For every Coxeter group $W$, the positive subcomplex $\tilde X_W^+$ is contractible.
    \label{thm:positive-subcomplex-contractible}
\end{theorem}

\begin{proof}
    Let $Z = \tilde X_W^+$ and recall the order-preserving map $\eta \colon \F(Z) \to G_W^+$ defined earlier.
    For every $\beta \in G_W^+ \setminus \{1\}$, the subcomplex $Z_{\lleq \beta}$ has a finite number of cells because each fiber is finite (as observed in the proof of \Cref{lemma2}) and there is a finite number of monoid elements that are $\lleq \beta$.

    Consider now the matching $\M = \bigcup_{\beta \in G_W^+ \setminus \{1\}} \M_\beta$ on $Z$.
    Thanks to \Cref{lemma1,lemma2} and the previous observation, we can apply \Cref{thm:patchwork} and conclude that $\M$ is a Morse matching.
    The only critical cell is the $0$-cell $\sigma_{(1, \varnothing)}$, so $Z$ is contractible by \Cref{thm:dmt}.
\end{proof}

\begin{corollary}
    The $K(\pi, 1)$ conjecture holds for all Artin groups of spherical type.
\end{corollary}

\begin{proof}
    Apply \Cref{lemma0} and \Cref{thm:positive-subcomplex-contractible} to conclude that the universal cover $\tilde X_W$ is contractible.
\end{proof}

In Artin monoids of non-spherical type, there is no common multiple of all standard generators, and thus, no Garside element. As a result, the proof of \Cref{lemma0} is not applicable, and the proof strategy outlined in the present notes is unlikely to be generalizable.
Alternative Garside structures, using a different generating set, have proven to be a fruitful avenue for studying certain Artin groups of non-spherical type. This has led, in particular, to the solution of the word problem and the $K(\pi, 1)$ conjecture in the affine case \cite{mccammond2017artin,paolini2021proof}.

\bibliographystyle{amsalpha-abbr}
\bibliography{bibliography}

\providecommand{\bysame}{\leavevmode\hbox to3em{\hrulefill}\thinspace}
\providecommand{\MR}{\relax\ifhmode\unskip\space\fi MR }
\providecommand{\MRhref}[2]{%
  \href{http://www.ams.org/mathscinet-getitem?mr=#1}{#2}
}
\providecommand{\href}[2]{#2}
\begin{thebibliography}{Hua24b}

\bibitem[Art25]{artin1925theorie}
E.~Artin, \emph{Theorie der Z{\"o}pfe}, Abh. Math. Sem. Univ. Hamburg (1925).

\bibitem[Bat02]{batzies2002discrete}
E.~Batzies, \emph{Discrete Morse theory for cellular resolutions}, Ph.D.
  thesis, 2002,
  \url{http://archiv.ub.uni-marburg.de/diss/z2002/0115/pdf/deb.pdf}.

\bibitem[BB06]{bjorner2006combinatorics}
A.~Bj{\"o}rner and F.~Brenti, \emph{Combinatorics of Coxeter groups}, vol. 231,
  Springer-Verlag, 2006.

\bibitem[Bou68]{bourbaki1968elements}
N.~Bourbaki, \emph{{\'E}l{\'e}ments de math{\'e}matique: Fasc. XXXIV. Groupes
  et alg{\`e}bres de Lie; Chap. 4, Groupes de Coxeter et syst{\`e}mes de Tits;
  Chap. 5; Chap. 6, Syst{\`e}mes de racines}, Hermann, 1968.

\bibitem[Bri71]{brieskorn1971fundamentalgruppe}
E.~Brieskorn, \emph{Die Fundamentalgruppe des Raumes der regul{\"a}ren Orbits
  einer endlichen komplexen Spiegelungsgruppe}, Inventiones Mathematicae
  \textbf{12} (1971), 57--61.

\bibitem[Bri73]{brieskorn1973groupes}
\bysame, \emph{Sur les groupes de tresses [d'apr{\`e}s VI Arnol'd]},
  S{\'e}minaire Bourbaki vol. 1971/72 Expos{\'e}s 400--417, Springer, 1973,
  pp.~21--44.

\bibitem[BS72]{brieskorn1972artin}
E.~Brieskorn and K.~Saito, \emph{Artin-gruppen und Coxeter-gruppen},
  Inventiones Mathematicae \textbf{17} (1972), no.~4, 245--271.

\bibitem[CD95]{charney1995k}
R.~Charney and M.~W. Davis, \emph{The $K(\pi, 1)$-problem for hyperplane
  complements associated to infinite reflection groups}, Journal of the
  American Mathematical Society (1995), 597--627.

\bibitem[Cha00]{chari2000discrete}
M.~K. Chari, \emph{On discrete Morse functions and combinatorial
  decompositions}, Discrete Mathematics \textbf{217} (2000), no.~1, 101--113.

\bibitem[CMS10]{callegaro2010k}
F.~Callegaro, D.~Moroni, and M.~Salvetti, \emph{The $K (\pi, 1)$ problem for
  the affine Artin group of type $B_n$ and its cohomology}, Journal of the
  European Mathematical Society \textbf{12} (2010), no.~1, 1--22.

\bibitem[Cox34]{coxeter1934discrete}
H.~S.~M. Coxeter, \emph{Discrete groups generated by reflections}, Annals of
  Mathematics (1934), 588--621.

\bibitem[Cox35]{coxeter1935complete}
H.~S. Coxeter, \emph{The complete enumeration of finite groups of the form
  $R_i^2 = (R_i R_j)^{k_{ij}}= 1$}, Journal of the London Mathematical Society
  \textbf{1} (1935), no.~1, 21--25.

\bibitem[Del72]{deligne1972immeubles}
P.~Deligne, \emph{Les immeubles des groupes de tresses
  g{\'e}n{\'e}ralis{\'e}s}, Inventiones Mathematicae \textbf{17} (1972), no.~4,
  273--302.

\bibitem[DPS24]{delucchi2022dual}
E.~Delucchi, G.~Paolini, and M.~Salvetti, \emph{Dual structures on Coxeter and
  Artin groups of rank three}, Geometry \& Topology \textbf{28} (2024), no.~9,
  4295--4336.

\bibitem[FN62]{fox1962braid}
R.~Fox and L.~Neuwirth, \emph{The braid groups}, Mathematica Scandinavica
  \textbf{10} (1962), 119--126.

\bibitem[For98]{forman1998morse}
R.~Forman, \emph{Morse theory for cell complexes}, Advances in Mathematics
  \textbf{134} (1998), no.~1, 90--145.

\bibitem[For02]{forman2002user}
\bysame, \emph{A user's guide to discrete Morse theory}, Séminaire
  Lotharingien de Combinatoire \textbf{48} (2002).

\bibitem[Hen85]{hendriks1985hyperplane}
H.~Hendriks, \emph{Hyperplane complements of large type}, Inventiones
  Mathematicae \textbf{79} (1985), no.~2, 375--381.

\bibitem[Her05]{hersh2005optimizing}
P.~Hersh, \emph{On optimizing discrete Morse functions}, Advances in Applied
  Mathematics \textbf{35} (2005), no.~3, 294--322.

\bibitem[HH25]{haettel2023new}
T.~Haettel and J.~Huang, \emph{New Garside structures and applications to Artin
  groups}, Duke Mathematical Journal \textbf{174} (2025), no.~9, 1665--1722.

\bibitem[Hua24a]{huang2024cycles}
J.~Huang, \emph{Cycles in spherical Deligne complexes and application to
  $K(\pi, 1)$-conjecture for Artin groups}, arXiv preprint arXiv:2405.12068
  (2024).

\bibitem[Hua24b]{huang2023labeled}
\bysame, \emph{Labeled four cycles and the $K(\pi, 1)$-conjecture for Artin
  groups}, Inventiones Mathematicae \textbf{238} (2024), no.~3, 905--994.

\bibitem[Hum92]{humphreys1992reflection}
J.~E. Humphreys, \emph{Reflection groups and Coxeter groups}, vol.~29,
  Cambridge university press, 1992.

\bibitem[Jon08]{jonsson2008simplicial}
J.~Jonsson, \emph{Simplicial complexes of graphs}, vol. 1928, Springer, 2008.

\bibitem[Koz07]{kozlov2007combinatorial}
D.~N. Kozlov, \emph{Combinatorial algebraic topology}, vol.~21,
  Springer-Verlag, 2007.

\bibitem[LP21]{lofano2021euclidean}
D.~Lofano and G.~Paolini, \emph{Euclidean matchings and minimality of
  hyperplane arrangements}, Discrete Mathematics \textbf{344} (2021), no.~3,
  112232.

\bibitem[MS17]{mccammond2017artin}
J.~McCammond and R.~Sulway, \emph{Artin groups of Euclidean type}, Inventiones
  Mathematicae \textbf{210} (2017), no.~1, 231--282.

\bibitem[Oko79]{okonek1979dask}
C.~Okonek, \emph{Das K($\pi$, 1)-Problem f{\"u}r die affinen Wurzelsysteme vom
  Typ $A_n$, $C_n$}, Mathematische Zeitschrift \textbf{168} (1979), no.~2,
  143--148.

\bibitem[Pao15]{paolini2015thesis}
G.~Paolini, \emph{Discrete Morse theory and the $K(\pi, 1)$ conjecture},
  Master's thesis, University of Pisa, 2015.

\bibitem[Par02]{paris2002artin}
L.~Paris, \emph{Artin monoids inject in their groups}, Commentarii Mathematici
  Helvetici \textbf{77} (2002), no.~3, 609--637.

\bibitem[Par14]{paris2012k}
\bysame, \emph{$K(\pi, 1)$ conjecture for Artin groups}, Annales de la
  Facult{\'e} des Sciences de Toulouse Math{\'e}matiques, vol.~23, 2014,
  pp.~361--415.

\bibitem[PS21]{paolini2021proof}
G.~Paolini and M.~Salvetti, \emph{Proof of the $K (\pi, 1)$ conjecture for
  affine Artin groups}, Inventiones Mathematicae \textbf{224} (2021), no.~2,
  487--572.

\bibitem[Sal87]{salvetti1987topology}
M.~Salvetti, \emph{Topology of the complement of real hyperplanes in
  $\mathbb{C}^N$}, Inventiones Mathematicae \textbf{88} (1987), no.~3,
  603--618.

\bibitem[Sal94]{salvetti1994homotopy}
\bysame, \emph{The homotopy type of Artin groups}, Mathematical Research
  Letters \textbf{1} (1994), no.~5, 565--577.

\bibitem[Tit61]{tits1961groupes}
J.~Tits, \emph{Groupes et g{\'e}om{\'e}tries de Coxeter}, Institut des Hautes
  {\'E}tudes Scientifiques (1961).

\bibitem[Tit69]{tits1969probleme}
\bysame, \emph{Le probleme des mots dans les groupes de Coxeter}, Symposia
  Mathematica, vol.~1, 1969, pp.~175--185.

\bibitem[VdL83]{van1983homotopy}
H.~Van~der Lek, \emph{The homotopy type of complex hyperplane complements},
  Ph.D. thesis, Katholieke Universiteit te Nijmegen, 1983.

\bibitem[Wit41]{witt1941spiegelungsgruppen}
E.~Witt, \emph{Spiegelungsgruppen und Aufz{\"a}hlung halbeinfacher Liescher
  Ringe}, Abhandlungen aus dem Mathematischen Seminar der Universit{\"a}t
  Hamburg, vol.~14, Springer, 1941, pp.~289--322.

\end{thebibliography}

\end{document}